\documentclass[12pt]{article}
\usepackage[top=27mm, bottom=27mm, left=22mm, right=22mm]{geometry}
\usepackage{lmodern}
\usepackage{amsmath,amssymb,amsthm}
\usepackage{microtype}
\usepackage[hidelinks]{hyperref}
\hypersetup{
  pdftitle={Thresholds and spread in set systems of bounded VC-dimension},
  pdfauthor={Chong Shangguan},
  pdfsubject={Research manuscript},
  pdfkeywords={expectation threshold, VC-dimension, spread measure, robust sunflower}
}

\newtheorem{theorem}{Theorem}[section]
\newtheorem{lemma}[theorem]{Lemma}
\newtheorem{corollary}[theorem]{Corollary}
\newtheorem{proposition}[theorem]{Proposition}
\newtheorem{claim}[theorem]{Claim}
\theoremstyle{remark}

\newcommand{\E}{\mathbb E}
\newcommand{\F}{\mathcal F}
\newcommand{\G}{\mathcal G}
\renewcommand{\Pr}{\mathop{\mathrm{Pr}}\nolimits}
\DeclareMathOperator{\Bin}{Bin}
\DeclareMathOperator{\VC}{VC}

\title{Thresholds and spread in set systems of bounded VC-dimension}
\author{Chong Shangguan\thanks{Research Center for Mathematics and Interdisciplinary Sciences,
Shandong University, Qingdao 266237, China, and Frontiers Science Center
for Nonlinear Expectations, Ministry of Education, Qingdao 266237, China.
Email: \mbox{\texttt{theoreming@163.com}}.}}
\date{}

\begin{document}
\maketitle

\begin{abstract}
Let $p_c(\mathcal F)$, $q(\mathcal F)$, and $q_f(\mathcal F)$ denote the threshold, expectation threshold, and fractional expectation threshold of a family $\mathcal F$ of nonempty subsets of a finite set, respectively. We prove that there is an absolute constant $C>0$ such that, if $\mathcal F$ has VC dimension at most $d\ge1$, then $p_c(\mathcal F)\le Cq(\mathcal F)\log(d+1)$. More generally, for every $0<\varepsilon\le1/2$, a binomial random set of density $\min\{1,Cq(\mathcal F)\log((d+1)/\varepsilon)\}$ contains a member of $\mathcal F$ with probability at least $1-\varepsilon$. Consequently, $q_f(\mathcal F)\le Cq(\mathcal F)\log(d+1)$, verifying Talagrand's integral--fractional conjecture for families of any fixed VC dimension.

We also prove that if a $k$-spread probability measure has support of VC dimension at most $d$, then a binomial random set of density $\min\{1,(C/k)\log((d+1)/\varepsilon)\}$ contains a member of its support with probability at least $1-\varepsilon$. In both random-containment results, the factor \(\log((d+1)/\varepsilon)\) is optimal up to absolute constants. As an application of the spread theorem, we prove that every $n$-uniform family of VC dimension at most $d$ with more than $(C p^{-1}\log((d+1)/\varepsilon))^n$ members contains a $(p,\varepsilon)$-robust sunflower. In particular, every such family with more than $(Cr\log(d+1))^n$ members contains an $r$-sunflower, improving the recent bound $(Crd)^n$ of Ge, Wang, Xu, and Zhao.
\end{abstract}


\section{Introduction}

A basic problem in probabilistic combinatorics is to understand the threshold at which a random subset of a finite ground set contains a member of a prescribed set system. We briefly recall the necessary definitions and background; for more details, see the recent surveys~\cite{ParkSurvey,ParkICM,Perkins,Lovett}.
Throughout the paper, $\log$ denotes the natural logarithm, and $\log^*$ denotes the iterated logarithm.

Let $U$ be a finite set, and write $W\sim\Bin(U,p)$ when each point of $U$ is included in $W$ independently with probability $p$. Let $\F\subseteq2^U\setminus\{\varnothing\}$ be nonempty. Its \emph{threshold} $p_c(\F)$ is the unique $p$ such that $\Pr_W[\exists S\in\F: S\subseteq W]=1/2$. 

For a family $\G\subseteq 2^U$, the \emph{span} of $\G$ is $\langle\G\rangle=\{A\subseteq U:\exists T\in\G,\ T\subseteq A\}$. We say that $\G$ \emph{covers} $\F$ if $\F\subseteq\langle\G\rangle$. If $\G$ covers $\F$, then every $W$ containing a member of $\F$ contains a member of $\G$. Hence $\Pr_W[\exists S\in\F: S\subseteq W]\le \sum_{T\in\G}p^{|T|}$. Optimizing this first-moment bound over all covers leads to the \emph{expectation threshold} $q(\F)$: it is the largest $p\in[0,1]$ for which some cover $\G$ satisfies $\sum_{T\in\G}p^{|T|}\le\frac12$. For future reference, we say that $\F$ is \emph{$p$-small} if it admits such a cover.

One may assign weights rather than choose an integral cover. A function $\lambda:2^U\to[0,\infty)$ is a \emph{fractional cover} for $\F$ if $\sum_{T\subseteq S}\lambda(T)\ge1$ for every $S\in\F$. The \emph{fractional expectation threshold} $q_f(\F)$ is the largest $p\in[0,1]$ for which there exists a fractional cover for $\F$ with $\sum_{T\subseteq U}\lambda(T)p^{|T|}\le\frac12$.

The definitions and the first-moment bound give \(q(\F)\le q_f(\F)\le p_c(\F)\); see~\cite{FKNP}. Relations among these three parameters have motivated several important conjectures. Let $\min(\F)$ denote the family of inclusion-minimal members of $\F$, and set $\ell(\F)=\max\{2,\max_{S\in\min(\F)}|S|\}$. Kahn and Kalai~\cite{KK} conjectured that there is an absolute constant $C>0$ such that \(p_c(\F)\le Cq(\F)\log\ell(\F)\). Talagrand~\cite{Talagrand} later proposed the fractional version of this conjecture, with $q_f(\F)$ in place of $q(\F)$. The fractional conjecture was proved by Frankston, Kahn, Narayanan and Park~\cite{FKNP}, while the original Kahn--Kalai conjecture was subsequently proved by Park and Pham~\cite{PP}.

Talagrand~\cite{Talagrand} also conjectured the stronger integral--fractional bound $q_f(\F)\le Cq(\F)$ with an absolute constant $C$. This remains open in general. Recently, Pham~\cite[Corollary~1.8]{Pham} obtained \(q_f(\F)\le Cq(\F)\log\log(4|U|)\), and Park~\cite{Park} subsequently proved a bound \(q_f(\F)\le Cq(\F)\max\{1,\log\log(1/q(\F))\}\) which is independent of the ground-set size.

We now turn to set systems of bounded VC dimension, focusing on bounds for the threshold in terms of the expectation threshold.

\subsection{Main results}

\paragraph{Thresholds versus expectation thresholds.}
The logarithmic dependence on $\ell(\F)$ in the Kahn--Kalai conjecture is necessary in general~\cite{KK}; see also the transversal example in Section~\ref{sec:thresholds}. It is therefore natural to ask whether additional structure of the set system can reduce or eliminate this dependence. 

A set $D\subseteq U$ is \emph{shattered} by $\F$ if $\{S\cap D:S\in\F\}=2^D$, and the \emph{VC dimension} $\VC(\F)$ is the largest size of a set shattered by $\F$. Balogh, Bernshteyn, Delcourt, Ferber and Pham~\cite{BBDFP} proved that if $\VC(\F)\le d$, then \(p_c(\F)\le Cq(\F)\bigl(\log(d+1)+\log^*\ell(\F)\bigr)\)
for an absolute constant $C>0$. Our first result removes the dependence on $\ell(\F)$ entirely.

\begin{theorem}\label{thm:main}
There is an absolute constant $C>0$ with the following property. Let $\F\subseteq2^U\setminus\{\varnothing\}$ be nonempty, and suppose that $\VC(\F)\le d$ for an integer $d\ge1$. For $0<\varepsilon\le1/2$, set
\begin{equation*}
p=\min\left\{1,Cq(\F)\log\frac{d+1}{\varepsilon}\right\}.
\end{equation*}
If $W\sim\Bin(U,p)$, then
\(\Pr_W[\exists S\in\F: S\subseteq W]\ge1-\varepsilon\).
In particular, $p_c(\F)\le Cq(\F)\log(d+1)$.
\end{theorem}

Since $q_f(\F)\le p_c(\F)$, Theorem~\ref{thm:main} also bounds the fractional expectation threshold in terms of the expectation threshold. In particular, Talagrand's integral--fractional conjecture~\cite{Talagrand} holds for set systems with VC dimension bounded by an absolute constant.

\begin{corollary}\label{cor:integral-fractional}
There is an absolute constant $C>0$ with the following property. Let $\F\subseteq2^U\setminus\{\varnothing\}$ be nonempty, and suppose that $\VC(\F)\le d$ for an integer $d\ge1$. Then $q(\F)\le q_f(\F)\le Cq(\F)\log(d+1)$.
\end{corollary}

We emphasize that the VC-dimension assumption is imposed on the family $\F$ being covered. Since \(p_c,q,q_f\) are unchanged by replacing \(\F\) with \(\min(\F)\), while \(\VC(\min(\F))\le\VC(\F)\), both results hold with \(\VC(\min(\F))\) in place of \(\VC(\F)\).

Recently, Ascoli, He, Park and Talagrand~\cite{AHPT} introduced the \emph{\(k\)-threshold} \(p_k(\mathcal I)\) and showed that Talagrand's discrete convexity conjecture~\cite{TalagrandConvexity} is equivalent to the existence of a universal \(k\ge2\) and an absolute constant \(L\) such that \(p_k(\mathcal I)\le Lq(\mathcal I)\) for every nontrivial increasing family \(\mathcal I\). Since \(p_k(\mathcal I)\le p_1(\mathcal I)=p_c(\mathcal I)\), Theorem~\ref{thm:main} verifies this bound when \(\VC(\min(\mathcal I))\) is bounded by an absolute constant; indeed, it already holds for \(k=1\).

\paragraph{A spread theorem under bounded VC dimension.}

A second viewpoint on threshold problems comes from \emph{spread measures}. Their relevance to the subject emerged from the work of Alweiss, Lovett, Wu and Zhang~\cite{ALWZ} on the Erd\H{o}s--Rado sunflower conjecture~\cite{ER}. Building on their ideas, Frankston, Kahn, Narayanan and Park~\cite{FKNP} proved the fractional Kahn--Kalai conjecture through a combination of linear-programming duality and spread methods; the subsequent proof of the Kahn--Kalai conjecture by Park and Pham~\cite{PP} was again inspired by this line of work, although its analysis avoids the use of spread.

For a real number $k\ge1$, a probability measure $\mu$ on $2^U$ is \emph{$k$-spread} if, for $S\sim\mu$, we have $\Pr_S[T\subseteq S]\le k^{-|T|}$ for every nonempty $T\subseteq U$. The \emph{support} of $\mu$ is $\{S\subseteq U:\mu(S)>0\}$.

Spreadness is closely tied to the fractional expectation threshold. Suppose that $\F$ supports a $k$-spread probability measure $\mu$. If $\lambda$ is any fractional cover of $\F$, then
\begin{equation*}
1\le \E_{S\sim\mu}\sum_{T\subseteq S}\lambda(T)
=\sum_{T\subseteq U}\lambda(T)\Pr_S[T\subseteq S]
\le\sum_{T\subseteq U}\lambda(T)k^{-|T|}.
\end{equation*}
Consequently $q(\F)\le q_f(\F)\le1/k$. Conversely, a linear-programming duality argument due to Talagrand (see~\cite[Proposition~1.5]{FKNP}), together with the trivial fact that every probability measure is $1$-spread, shows that $\F$ supports a $\max\{1,(2q_f(\F))^{-1}\}$-spread probability measure. Thus, up to a factor of two, $q_f(\F)$ is the reciprocal of the largest spread parameter supported by $\F$.

Spread theorems convert this structural information into random containment. A family $\F\subseteq2^U$ is \emph{$n$-bounded} if $|S|\le n$ for every $S\in\F$. If a $k$-spread probability measure has $n$-bounded support $\F$, then the standard spread theorem (see~\cite{ALWZ,Bell}) implies that a random set $W\sim\Bin(U,p)$ contains a member of $\F$ with probability at least $1-\varepsilon$ once $p$ is of order $k^{-1}\log(n/\varepsilon)$. The bounded-VC threshold theorem of Balogh et al.~\cite[Theorem~1.6]{BBDFP} implies the improved bound $p=O(k^{-1}(\log(d/\varepsilon)+\log^*n))$ when $\VC(\F)\le d$; for $d=1$, their Theorem~1.7 already removes the dependence on $n$ entirely. Our spread theorem eliminates this dependence for every $d$.

\begin{theorem}\label{thm:spread}
There is an absolute constant $C>0$ with the following property. Let $\mu$ be a $k$-spread probability measure on $2^U$, where $k\ge1$, whose support $\F$ satisfies $\VC(\F)\le d$ for an integer $d\ge1$. For $0<\varepsilon\le1/2$, set
\begin{equation*}
p=\min\left\{1,\frac Ck\log\frac{d+1}{\varepsilon}\right\}.
\end{equation*}
If $W\sim\Bin(U,p)$, then $\Pr_W[\exists S\in\F:S\subseteq W]\ge1-\varepsilon$.
\end{theorem}

Although Theorem~\ref{thm:spread} follows formally from Theorem~\ref{thm:main}, since a $k$-spread measure gives $q(\F)\le1/k$, our proof proceeds in the opposite direction. We prove Theorem~\ref{thm:spread} first and use it to derive Theorem~\ref{thm:main}.

\paragraph{Sharpness.}
The dependence on the parameters in our two main theorems is essentially sharp. A standard transversal construction for spread families~\cite{ALWZ,BCW} shows that the dependence on $k$, $d$ and $\varepsilon$ in Theorem~\ref{thm:spread} is optimal up to absolute constants, and also shows that the factor $\log(d+1)$ in Theorem~\ref{thm:main} is necessary. We give the details in Section~\ref{sec:thresholds}. This example has $q(\F)=q_f(\F)$, so it does not establish the necessity of the $\log(d+1)$ factor in Corollary~\ref{cor:integral-fractional}.

\subsection{Applications}

We next highlight two applications of the preceding threshold and spread results.

\paragraph{Sunflowers.}
Sunflowers provide a natural application of spread methods, which played a central role in recent progress on the sunflower problem~\cite{ALWZ}. Distinct sets $S_1,\ldots,S_r$ form an \emph{$r$-sunflower} if $S_i\cap S_j$ is the same set for every $i\ne j$. Fox, Pach and Suk~\cite{FPS} initiated the study of the sunflower problem under a VC-dimension restriction. For $d\ge2$, Balogh et al.~\cite{BBDFP} proved that every $n$-uniform family $\F$ with $\VC(\F)\le d$ and $|\F|>(Cr(\log d+\log^*n))^n$ contains an $r$-sunflower. More recently, Ge, Wang, Xu and Zhao~\cite{GWXZ} removed the dependence on $n$ from the exponential base, proving that $|\F|>(50dr)^n$ suffices.

Robust sunflowers were introduced by Rossman~\cite{Rossman} under the name \emph{quasi-sunflowers}, in connection with lower bounds for monotone circuits. They subsequently became an important ingredient in the modern approach to the sunflower problem; see, for example,~\cite{ALWZ,Lovett}. The general robust sunflower lemma gives an exponential-base bound of order $p^{-1}\log(n/\varepsilon)$ for $n$-uniform families. Our spread theorem replaces this dependence on the member size by a logarithmic dependence on the VC dimension.

A family $\mathcal R\subseteq2^U$ with at least two members is called a \emph{$(p,\varepsilon)$-robust sunflower} (see, e.g.,~\cite{ALWZ,Lovett}) if, writing $T=\bigcap_{R\in\mathcal R}R$, we have $T\notin\mathcal R$ and $\Pr_W[\exists R\in\mathcal R:R\subseteq T\cup W]\ge1-\varepsilon$ for $W\sim\Bin(U,p)$. The set $T$ is called the \emph{core}. Thus every \emph{petal} $R\setminus T$ is nonempty. The petals need not be pairwise disjoint, and a binomial random set of density $p$ contains at least one of them with probability at least $1-\varepsilon$.

A maximal-core argument turns a sufficiently large uniform family into a spread family of petals, to which Theorem~\ref{thm:spread} applies. In the bounded-VC setting, this yields a robust sunflower bound with exponential base of order $p^{-1}\log((d+1)/\varepsilon)$, independent of the member size.

\begin{corollary}\label{cor:robust}
There is an absolute constant $C>0$ such that, for integers $d,n\ge1$, $0<p\le1$ and $0<\varepsilon\le1/2$, every $n$-uniform family $\F\subseteq2^U$ with $\VC(\F)\le d$ and
\begin{equation*}
|\F|>\left(\frac Cp\log\frac{d+1}{\varepsilon}\right)^n
\end{equation*}
contains a $(p,\varepsilon)$-robust sunflower.
\end{corollary}

Taking $p=1/(2r)$ and $\varepsilon=1/2$, and then using a random coloring of the petals, gives the ordinary sunflower consequence that every $n$-uniform family $\F$ with $\VC(\F)\le d$ and $|\F|>(Cr\log(d+1))^n$ contains an $r$-sunflower. Thus the exponential base is independent of the member size, while the dependence on the VC dimension improves from the linear bound in~\cite{GWXZ} to logarithmic.

\paragraph{Families with bounded intersections.}
Another application comes from set systems with small pairwise intersections. If $\F\subseteq2^U\setminus\{\varnothing\}$ satisfies $|A\cap B|\le s$ for all distinct $A,B\in\F$, then $\VC(\F)\le s+1$. Hence the result of Balogh et al.~\cite{BBDFP} gives $p_c(\F)\le Cq(\F)(\log(s+2)+\log^*\ell(\F))$. Our theorem removes the dependence on the maximum member size entirely.

\begin{corollary}\label{cor:bounded-intersections}
There is an absolute constant $C>0$ such that, for every integer $s\ge0$, if $\F\subseteq2^U\setminus\{\varnothing\}$ is nonempty and $|A\cap B|\le s$ for all distinct $A,B\in\F$, then $p_c(\F)\le Cq(\F)\log(s+2)$.
\end{corollary}

In particular, this implies that every nonempty linear hypergraph without an empty edge satisfies $p_c=\Theta(q)$ with absolute constants. 

\paragraph{Organization of the paper.}
Section~\ref{sec:weighted-traces} proves a weighted trace estimate and derives its common-intersection and independent-sample consequences. Section~\ref{sec:fragments} uses the common-intersection estimate and a theorem of Bell~\cite{Bell} to prove Theorem~\ref{thm:spread}. Section~\ref{sec:thresholds} combines maximum packing with the independent-sample estimate to prove Theorem~\ref{thm:main} and gives the sharpness example. Section~\ref{sec:applications} derives the sunflower consequences. Each section begins with a brief roadmap outlining its main result and the key ideas of the argument.

\section{Trace estimates for spread measures with bounded-VC support}\label{sec:weighted-traces}

Throughout this section, $k\ge8$, $d$ is a positive integer, and $\mu$ is a $k$-spread probability measure on $2^U$ with support $\F$ satisfying $\VC(\F)\le d$. If $S\in\F$ is sampled according to $\mu$, we write $S\sim\mu$. For a set $E\subseteq U$, the \emph{trace} of $\F$ on $E$ is the family $\F|_E=\{S\cap E:S\in\F\}$.
If $S$ is sampled from a probability measure $\mu$ supported on $\F$, then $\mu$ induces a probability distribution on $\F|_E$, assigning to each $A\in\F|_E$ the probability $\Pr_S[S\cap E=A]$. For every nonnegative integer $a$, define
\begin{equation*}
\Pi_d(a):=\sum_{i=0}^{\min\{d,a\}}\binom ai.
\end{equation*}
By the Sauer--Shelah lemma~\cite{Sauer,Shelah}, if $E\subseteq U$ has size $a$, then $\F|_E$ has at most $\Pi_d(a)$ members. We will use the standard estimates
\begin{equation}\label{eq:sauer}
\Pi_d(a)\le \left[e\left(1+\frac ad\right)\right]^d,
\end{equation}
and, when $a\ge d$, $\Pi_d(a)\le(ea/d)^d$. As usual, $\binom am=0$ for integers $m>a\ge0$.

The proofs of Theorems~\ref{thm:main} and~\ref{thm:spread} rely on a common estimate for the intersection of $S$ with an auxiliary random set. The first ingredient is a version of \cite[Lemma~2.3]{GWXZ} for general spread probability measures: spreadness together with bounded VC dimension implies that the set of points with large marginal inclusion probability is small. We combine this with a tail bound for the auxiliary random set and then group traces according to their probabilities. This yields an exponential bound for the distribution of the intersection size and, consequently, for its binomial moments (see Proposition~\ref{prop:weighted-trace}).

We will apply the resulting estimate in two settings. When the auxiliary set is the common intersection $C(Z)$ defined below (see Corollary~\ref{cor:common-intersection}), it gives the estimate needed for the minimum-fragment argument in Section~\ref{sec:fragments}. When the auxiliary set is an independent copy of $S$ (see Corollary~\ref{cor:independent-intersection}), it controls the intersection of two independent samples and will be used in Section~\ref{sec:thresholds}.

\subsection{A weighted trace estimate}

The following proposition isolates the tail condition on the auxiliary random set \(D\) needed in both applications.


\begin{proposition}\label{prop:weighted-trace}
There is an absolute constant $C_{\mathrm{tr}}\ge1$ with the following property. Let $k\ge8$, let $d\ge1$ be an integer, and let $\mu$ be a $k$-spread probability measure on $2^U$ whose support $\F$ satisfies $\VC(\F)\le d$. Let $S\sim\mu$, and let $D\subseteq U$ be a random set independent of $S$. Suppose that, for some $0<p\le k^{-1/8}$,
\begin{equation}\label{eq:trace-tail}
\Pr_D[|D\cap E|\ge r]\le\Pi_d(|E|)p^r
\end{equation}
for every $E\subseteq U$ and every integer $r\ge1$. For a fixed realization of $D$ and a trace $A\in\F|_D$, write $w_D(A)=\Pr_S[S\cap D=A]$.
Then, for every integer $a\ge1$,
\begin{equation}\label{eq:weighted-trace}
\Pr_{S,D}[|S\cap D|=a]
=\E_D\sum_{A\in\F|_D,\,|A|=a}w_D(A)
\le (C_{\mathrm{tr}}d)^d\left(\frac2k\right)^a.
\end{equation}
Consequently, for every integer $m\ge1$,
\begin{equation}\label{eq:intersection-bound}
\E_{S,D}\binom{|S\cap D|}{m}
\le (C_{\mathrm{tr}}d)^d\left(\frac4k\right)^m.
\end{equation}
\end{proposition}

We call this a weighted trace estimate because, instead of merely counting the traces $A\in\F|_D$, we weight each trace by $w_D(A)=\Pr_S[S\cap D=A]$ and bound the expected total weight of the traces of each fixed size.

\begin{proof}
Throughout the proof, $C>0$ denotes an absolute constant that may change from line to line.
We first establish a version of \cite[Lemma~2.3]{GWXZ} for general spread probability measures. 
For $0<\lambda\le1$, put $U_\lambda=\{x\in U:\Pr_S[x\in S]\ge\lambda\}$.

\begin{claim}\label{claim1}
For every $0<\lambda\le1$,
\begin{equation}\label{eq:frequent}
|U_\lambda|\le\frac{4d}{\lambda^2}.
\end{equation}
\end{claim}

\begin{proof}
Fix $E\subseteq U$. If $S\cap E=A$ and $A\ne\varnothing$, then $A\subseteq S$, so spreadness gives $\Pr_S[S\cap E=A]\le k^{-|A|}$; for $A=\varnothing$, the same bound is trivial. Since there are at most $\Pi_d(|E|)$ possible traces on $E$, Jensen's inequality gives
\begin{equation*}
k^{\E_S|S\cap E|}\le \E_S k^{|S\cap E|}
=\sum_{A\in\F|_E}k^{|A|}\Pr_S[S\cap E=A]
\le\Pi_d(|E|).
\end{equation*}
Taking logarithms and setting $E=U_\lambda$, if $|U_\lambda|\ge d$ then
\begin{equation*}
|U_\lambda|\lambda\log k
\le \E_S|S\cap U_\lambda|\log k
\le d\log\frac{e|U_\lambda|}{d}
\le2\sqrt{d|U_\lambda|},
\end{equation*}
where $1+\log u\le2\sqrt u$ for $u\ge1$. Since $\log k>1$, this implies \eqref{eq:frequent}. If $|U_\lambda|<d$, the same bound is immediate.
\end{proof}

We next combine the tail assumption on \(D\) with the VC bound to obtain a moment bound for \(|D\cap U_\lambda|\), which in turn controls the expected number of traces of \(\F\) on \(D\cap U_\lambda\).

\begin{claim}\label{claim2}
There is an absolute constant $C_1$ such that, for every $0<\lambda\le1$,
\begin{equation*}
\E_D\Pi_d(|D\cap U_\lambda|)
\le\left(C_1\left(1+\frac{\log(1/\lambda)}{\log(1/p)}\right)\right)^d.
\end{equation*}
\end{claim}

\begin{proof}
The case $E=\varnothing$ is trivial, so fix a nonempty set $E\subseteq U$ and set $r_0=\lceil\log\Pi_d(|E|)/\log(1/p)\rceil$. By \eqref{eq:trace-tail}, for every integer $s\ge0$,
\begin{equation*}
\Pr_D[|D\cap E|\ge r_0+s]
\le \Pi_d(|E|)p^{r_0+s}
\le p^s.
\end{equation*}
Thus, beyond \(r_0\), the probability that \(|D\cap E|\) exceeds \(r_0+s\) decays geometrically in \(s\).

Since $p\le k^{-1/8}$ and $k\ge8$, we have $p^s\le e^{-(\log8)s/8}$. Splitting according to whether $|D\cap E|<r_0$ and then summing over the possible values above $r_0$, the preceding tail bound gives
\begin{align*}
\E_D(|D\cap E|+d)^d
&\le (r_0+d)^d+\sum_{s\ge0}(r_0+d+s)^d e^{-(\log8)s/8}\\
&\le (r_0+d)^d+\sum_{j=0}^d\binom dj(r_0+d)^{d-j}
       \sum_{s\ge0}s^j e^{-(\log8)s/8}\\
&\le \bigl(C(r_0+d)\bigr)^d.
\end{align*}
Indeed, $\sum_{s\ge0}s^j e^{-(\log8)s/8}\le C^{j+1}j!\le C^{j+1}d^j$ for $0\le j\le d$, while $d^j\le(r_0+d)^j$.

Using \eqref{eq:sauer}, we therefore obtain
\begin{align}
\E_D\Pi_d(|D\cap E|)
&\le\left(\frac ed\right)^d\E_D(|D\cap E|+d)^d
\le\left(C\left(1+\frac{r_0}{d}\right)\right)^d \notag\\
&\le\left(C\left(1+\frac{\log\Pi_d(|E|)}{d\log(1/p)}\right)\right)^d,
\label{eq:trace-complexity}
\end{align}
where the last inequality uses $r_0\le1+\log\Pi_d(|E|)/\log(1/p)$ and $\log(1/p)\ge(\log8)/8$.

We now take $E=U_\lambda$. By Claim~\ref{claim1} and \eqref{eq:sauer}, we have
\begin{equation*}
1+\frac{|U_\lambda|}{d}\le\frac5{\lambda^2},
\qquad
\frac1d\log\Pi_d(|U_\lambda|)\le1+\log5+2\log(1/\lambda).
\end{equation*}
Substituting this into \eqref{eq:trace-complexity}, and using again that $\log(1/p)\ge(\log8)/8$, proves the claim after adjusting the absolute constant $C_1$.
\end{proof}

For fixed \(D\) and \(A\in\F|_D\), the event $S\cap D=A$ implies $A\subseteq S$, and hence $w_D(A)\le k^{-|A|}$. Moreover, if $w_D(A)\ge\lambda$, then every $x\in A$ satisfies $\Pr_S[x\in S]\ge w_D(A)\ge\lambda$, so $A\subseteq U_\lambda$. If $A=S\cap D$, this inclusion gives $A=S\cap(D\cap U_\lambda)$. Distinct such \(A\) are therefore distinct traces on \(D\cap U_\lambda\), so
\begin{equation}\label{eq:weighted-count}
\bigl|\{A\in\F|_D:w_D(A)\ge\lambda\}\bigr|
\le\Pi_d(|D\cap U_\lambda|).
\end{equation}

Unlike~\cite{GWXZ}, which groups points by their marginal probabilities, we group traces by \(w_D(A)\). Together with \eqref{eq:weighted-count} and Claim~\ref{claim2}, this controls the total probability of traces of each size.

Fix an integer $a\ge1$. For $j\ge0$, let
\begin{equation*}
\mathcal A_j
=\{A\in\F|_D:|A|=a,\ k^{-a}e^{-(j+1)}<w_D(A)\le k^{-a}e^{-j}\}.
\end{equation*}
Since $w_D(A)\le k^{-a}$ whenever $|A|=a$, the following two inequalities hold
\begin{align*}
\sum_{A\in\F|_D,\,|A|=a}w_D(A)&\le k^{-a}\sum_{j\ge0}e^{-j}|\mathcal A_j|,\\
|\mathcal A_j|&\le\bigl|\{A\in\F|_D:w_D(A)\ge k^{-a}e^{-(j+1)}\}\bigr|.
\end{align*}
Hence, by \eqref{eq:weighted-count} and Claim~\ref{claim2},
\begin{equation*}
\E_D\sum_{A\in\F|_D,\,|A|=a}w_D(A)
\le C_1^d k^{-a}\sum_{j\ge0}e^{-j}
\left(1+\frac{a\log k+j+1}{\log(1/p)}\right)^d.
\end{equation*}
Since $\log k/\log(1/p)\le8$ and $1/\log(1/p)\le8/\log8$, this sum is at most $C^d\sum_{j\ge0}e^{-j}(a+j+1)^d$. By the binomial theorem,
\begin{align*}
\sum_{j\ge0}e^{-j}(a+j+1)^d
&=\sum_{r=0}^d\binom dr a^{d-r}\sum_{j\ge0}e^{-j}(j+1)^r\\
&\le e^2\sum_{r=0}^d\binom dr a^{d-r}d^r
=e^2(a+d)^d
\le(C(a+d))^d,
\end{align*}
where we used $\sum_{j\ge0}e^{-j}(j+1)^r\le e^2r!\le e^2d^r$ for $0\le r\le d$. Hence $\E_D\sum_{A\in\F|_D,\,|A|=a}w_D(A)\le (C(a+d))^d k^{-a}$. Since $(a+d)^d=d^d(1+a/d)^d\le d^d2^{d+a}=(2d)^d2^a$, we obtain
\begin{equation*}
\E_D\sum_{A\in\F|_D,\,|A|=a}w_D(A)
\le(Cd)^d\left(\frac2k\right)^a.
\end{equation*}
Choosing $C_{\mathrm{tr}}$ sufficiently large proves \eqref{eq:weighted-trace}.

Finally, by \eqref{eq:weighted-trace} and $\binom am\le2^a$, we have
\begin{align*}
\E_{S,D}\binom{|S\cap D|}{m}
&=\sum_{a\ge m}\binom am\Pr_{S,D}[|S\cap D|=a]
\le (C_{\mathrm{tr}}d)^d\sum_{a\ge m}\binom am\left(\frac2k\right)^a\\
&\le (C_{\mathrm{tr}}d)^d\sum_{a\ge m}\left(\frac4k\right)^a
\le 2(C_{\mathrm{tr}}d)^d\left(\frac4k\right)^m.
\end{align*}
Since $d\ge1$, enlarging $C_{\mathrm{tr}}$ by an absolute factor absorbs the last factor $2$ and proves \eqref{eq:intersection-bound}.
\end{proof}

We fix such an absolute constant $C_{\mathrm{tr}}$ for the remainder of the paper.

\subsection{Common intersections and independent samples}

We first verify \eqref{eq:trace-tail} when the auxiliary set is a common intersection. Let $Z\subseteq U$. If $Z$ contains at least one member of $\F$, define its \emph{common intersection} by
\begin{equation*}
C(Z)=\bigcap_{F\in\F,\,F\subseteq Z}F;
\end{equation*}
if $Z$ contains no member of $\F$, set $C(Z)=\varnothing$. Thus $C(Z)$ consists of the points that are forced to lie in every member of $\F$ contained in $Z$.

\begin{corollary}\label{cor:common-intersection}
Let $k\ge8$, let $d\ge1$ be an integer, and let $\mu$ be a $k$-spread probability measure on $2^U$ whose support $\F$ satisfies $\VC(\F)\le d$. Let $S\sim\mu$, and let $Z\subseteq U$ be a random set independent of $S$. Suppose that, for some $0<p\le k^{-1/8}$,
\begin{equation}\label{eq:containment}
\Pr_Z[B\subseteq Z]\le p^{|B|}
\qquad\text{for every }B\subseteq U.
\end{equation}
Then, for every integer $m\ge1$,
\begin{equation}\label{eq:common-intersection}
\E_{S,Z}\binom{|S\cap C(Z)|}{m}
\le(C_{\mathrm{tr}}d)^d\left(\frac4k\right)^m.
\end{equation}
\end{corollary}

\begin{proof}
Fix $E\subseteq U$ and an integer $r\ge1$. If $|C(Z)\cap E|\ge r$, then $Z$ contains a member of $\F$, and every such member $F\subseteq Z$ contains $C(Z)$. Hence $|F\cap E|\ge r$, so some trace $A\in\F|_E$ with $|A|\ge r$ is contained in $Z$. Therefore, by \eqref{eq:containment} and the Sauer--Shelah bound,
\begin{align*}
\Pr_Z[|C(Z)\cap E|\ge r]
&\le \sum_{A\in\F|_E,\,|A|\ge r}\Pr_Z[A\subseteq Z]
\le \sum_{A\in\F|_E,\,|A|\ge r}p^{|A|}\\
&\le |\F|_E|\,p^r
\le \Pi_d(|E|)p^r.
\end{align*}
Thus \eqref{eq:trace-tail} holds with $D=C(Z)$. Since $C(Z)$ is determined by $Z$, it is independent of $S$, and Proposition~\ref{prop:weighted-trace} gives \eqref{eq:common-intersection}.
\end{proof}

We next consider the case in which the auxiliary set is an independent copy of $S$.

\begin{corollary}\label{cor:independent-intersection}
Let $k\ge8$, let $d\ge1$ be an integer, and let $\mu$ be a $k$-spread probability measure on $2^U$ whose support $\F$ satisfies $\VC(\F)\le d$. Let $S,R\sim\mu$ be independent. Then, for every integer $m\ge1$,
\begin{equation*}
\sum_{T\subseteq U,\,|T|=m}\Pr_S[T\subseteq S]^2
=\E_{S,R}\binom{|S\cap R|}{m}
\le(C_{\mathrm{tr}}d)^d\left(\frac4k\right)^m.
\end{equation*}
\end{corollary}

\begin{proof}
Fix $E\subseteq U$ and an integer $r\ge1$. Since $R\cap E$ is always a trace of $\F$ on $E$, spreadness and the Sauer--Shelah bound give
\begin{align*}
\Pr_R[|R\cap E|\ge r]
&=\sum_{A\in\F|_E,\,|A|\ge r}\Pr_R[R\cap E=A]
\le\sum_{A\in\F|_E,\,|A|\ge r}\Pr_R[A\subseteq R]\\
&\le\sum_{A\in\F|_E,\,|A|\ge r}k^{-|A|}
\le |\F|_E|\,k^{-r}
\le \Pi_d(|E|)k^{-r}.
\end{align*}
Thus \eqref{eq:trace-tail} holds with $D=R$ and $p=1/k$, so Proposition~\ref{prop:weighted-trace} gives $\E_{S,R}\binom{|S\cap R|}{m}\le(C_{\mathrm{tr}}d)^d(4/k)^m$. 

Finally, counting the \(m\)-subsets of \(S\cap R\) and using independence gives
\begin{align*}
\E_{S,R}\binom{|S\cap R|}{m}
&=\sum_{T\subseteq U,\,|T|=m}\Pr_{S,R}[T\subseteq S\cap R]\\
&=\sum_{T\subseteq U,\,|T|=m}\Pr_S[T\subseteq S]\Pr_R[T\subseteq R]
=\sum_{T\subseteq U,\,|T|=m}\Pr_S[T\subseteq S]^2.
\end{align*}
\end{proof}

\section{The bounded-VC spread theorem}\label{sec:fragments}



We prove Theorem~\ref{thm:spread} using two independent random exposures. The first maps a sampled member $S\in\F$ to a minimum fragment $M(S,V)$ outside a random set $V$. The common-intersection estimate from Section~\ref{sec:weighted-traces} gives an exponential tail for the fragment size, and conditioning on short fragments yields a $(k/2)$-spread probability measure supported on a bounded-size family. Bell's theorem~\cite{Bell} for bounded-size families (see Lemma~\ref{lem:bell}), applied to a second random exposure, then ensures that the union of the two exposures contains a member of \(\F\).

The minimum-fragment setup is closely related to the framework used by Balogh et al.~\cite{BBDFP} in their proof of the bounded-VC threshold theorem, building on the minimum-fragment method of Park and Pham~\cite{PP} (see also~\cite{Lovett}). Their argument controls large fragments by a VC-dimension double-counting estimate and then proceeds by induction on the maximum member size. Here the common-intersection estimate gives, in one step, fragments whose size depends only on $d$ and the desired error probability, so no induction on the original maximum member size is needed.

Let $\F\subseteq2^U$ be nonempty, and fix a total order on $\F$. For $S\in\F$ and $V\subseteq U$, choose a member $F\in\F$ with $F\subseteq S\cup V$ minimizing $|F\setminus V|$, breaking ties according to this order, and define its \emph{minimum fragment} by
\begin{equation*}
M(S,V)=F\setminus V.
\end{equation*}
The choice exists because $S$ itself is eligible. The two properties we will use are $M(S,V)\subseteq S\setminus V$ and $F\subseteq V\cup M(S,V)$. Thus $M(S,V)$ is the part of $F$ not already contained in $V$: if a second random set contains $M(S,V)$, then its union with $V$ contains $F$, and hence a member of $\F$. 

The next lemma analyzes the first random exposure $V\sim\Bin(U,16/k)$. As a key consequence of the common-intersection estimate from Section~\ref{sec:weighted-traces}, it gives an exponential tail for the size of the resulting minimum fragment.

\begin{lemma}\label{lem:fragment}
    Let $\mu$ be a $k$-spread probability measure on $2^U$ with support $\F$ of VC dimension at most $d$, where $k\ge34$ and $d\ge1$ is an integer, and let $S\sim\mu$. If $V\sim\Bin(U,16/k)$ is independent of $S$, then, for every integer $t\ge1$,
\begin{equation*}
    \Pr_{S,V}[|M(S,V)|\ge t]\le\frac43(C_{\mathrm{tr}}d)^d4^{-t}.
\end{equation*}
\end{lemma}

\begin{proof}
Put $\rho=16/k$. Given $S,V$, set $Z=V\cup M(S,V)$. Let $F\in\F$ be the member chosen in the definition of $M(S,V)$, so that $M(S,V)=F\setminus V$. It follows that $F\subseteq Z$.

Now let $F'\in\F$ be any member with $F'\subseteq Z$. Since $F'\subseteq Z\subseteq S\cup V$, the set $F'$ is eligible in the definition of $M(S,V)$, while $F'\setminus V\subseteq Z\setminus V=M(S,V)$. By minimality, we must have $F'\setminus V=M(S,V)$. Thus every member of $\F$ contained in $Z$ contains $M(S,V)$. By the definition of the common intersection $C(Z)$, we have
\begin{equation*}
M(S,V)\subseteq S\cap C(Z),\qquad V=Z\setminus M(S,V).
\end{equation*}

Fix an integer $m\ge1$. Expanding the probability according to the choices of $S$ and $V$ gives
\begin{equation*}
\Pr_{S,V}[|M(S,V)|=m]=\sum_{S\in\F}\mu(S)\sum_{V\subseteq U,\, |M(S,V)|=m}\Pr[V],
\end{equation*}
where, as usual, \(V\) denotes both the random set and a realization of it.
For each pair $(S,V)$ occurring in this sum, set $Z=V\cup M(S,V)$. Then $M(S,V)\subseteq S\cap C(Z)$, and $V$ is uniquely determined by $Z$ and $M(S,V)$ through $V=Z\setminus M(S,V)$. Therefore, for fixed $S$ and $Z$, the possible fragments of size $m$ are among the $m$-subsets of $S\cap C(Z)$. Enlarging the sum to all such subsets gives
\begin{equation*}
\Pr_{S,V}[|M(S,V)|=m]\le\sum_{S\in\F}\mu(S)\sum_{Z\subseteq U}\sum_{A\subseteq S\cap C(Z),\, |A|=m}\Pr_V[V=Z\setminus A].
\end{equation*}

For every $A\subseteq Z$ with $|A|=m$, since $V\sim\Bin(U,\rho)$,
\begin{equation*}
\Pr_V[V=Z\setminus A]=\rho^{|Z|-m}(1-\rho)^{|U|-|Z|+m}=\left(\frac{1-\rho}{\rho}\right)^m\rho^{|Z|}(1-\rho)^{|U|-|Z|}.
\end{equation*}
The product $\rho^{|Z|}(1-\rho)^{|U|-|Z|}$ is precisely the probability that a $\Bin(U,\rho)$ random set is equal to $Z$. Hence
\begin{align*}
\Pr_{S,V}[|M(S,V)|=m]
&\le\left(\frac{1-\rho}{\rho}\right)^m\sum_{S\in\F}\mu(S)\sum_{Z\subseteq U}\rho^{|Z|}(1-\rho)^{|U|-|Z|}\binom{|S\cap C(Z)|}{m}\\
&=\left(\frac{1-\rho}{\rho}\right)^m\E_{S,Z}\binom{|S\cap C(Z)|}{m},
\end{align*}
where on the right $S\sim\mu$ and $Z\sim\Bin(U,\rho)$ are independent.

Since $Z\sim\Bin(U,\rho)$, we have $\Pr_Z[B\subseteq Z]=\rho^{|B|}$ for every $B\subseteq U$. Moreover, $\rho=16/k\le k^{-1/8}$ when $k\ge34$. Hence Corollary~\ref{cor:common-intersection}, applied with $p=\rho$, gives
\begin{equation*}
\E_{S,Z}\binom{|S\cap C(Z)|}{m}\le(C_{\mathrm{tr}}d)^d\left(\frac4k\right)^m.
\end{equation*}
Since $(1-\rho)/\rho\le1/\rho=k/16$, it follows that
\begin{align*}
\Pr_{S,V}[|M(S,V)|=m]
&\le\left(\frac{1-\rho}{\rho}\right)^m(C_{\mathrm{tr}}d)^d\left(\frac4k\right)^m\\
&\le(C_{\mathrm{tr}}d)^d\left(\frac{k}{16}\cdot\frac4k\right)^m=(C_{\mathrm{tr}}d)^d4^{-m}.
\end{align*}
Finally, summing over $m\ge t$ gives
\begin{equation*}
\Pr_{S,V}[|M(S,V)|\ge t]\le(C_{\mathrm{tr}}d)^d\sum_{m\ge t}4^{-m}=\frac43(C_{\mathrm{tr}}d)^d4^{-t},
\end{equation*}
as required.
\end{proof}

We next use the following form of Bell's theorem for bounded-size families. Recall that, for \(0<a\le1\), a family \(\mathcal H\subseteq2^U\) is \(a\)-small if it admits a cover \(\mathcal A\) satisfying \(\sum_{A\in\mathcal A}a^{|A|}\le1/2\).

\begin{lemma}[{Bell~\cite[Theorem~3]{Bell}}]\label{lem:bell}
There is an absolute constant $c\ge2$ such that the following holds. Let $\mathcal H\subseteq2^U$ be $t$-bounded, where $t\ge1$, and let $0<a\le1$. Suppose that $\mathcal H$ is not $a$-small. For $0<\delta\le1/2$, set $p=\min\{1,ca\log(t/\delta)\}$. If $W\sim\Bin(U,p)$, then $\Pr_W[\exists A\in\mathcal H:A\subseteq W]\ge1-\delta$.
\end{lemma}

The cases \(p=1\) and \(\varnothing\in\mathcal H\) are immediate. In the remaining case, the assertion follows from Bell's Theorem~3 using that \(\mathcal H\) is \(t\)-bounded and monotonicity, with the absolute constant adjusted for our use of natural logarithms.

Suppose that $\nu$ is an $h$-spread probability measure supported on $\mathcal H$, and let $S\sim\nu$. Then every cover $\mathcal A$ of $\mathcal H$ satisfies
\begin{equation*}
1\le\sum_{A\in\mathcal A}\Pr_S[A\subseteq S]\le\sum_{A\in\mathcal A}h^{-|A|}.
\end{equation*}
Thus $\mathcal H$ is not $(1/h)$-small, and Lemma~\ref{lem:bell} applies with $a=1/h$.

\begin{proof}[Proof of Theorem~\ref{thm:spread}]
If $p=1$, the assertion is immediate. Hence assume $p<1$. In this case $p=(C/k)\log((d+1)/\varepsilon)$, and since $\log((d+1)/\varepsilon)\ge\log4$, taking the absolute constant $C$ sufficiently large ensures that $k\ge34$.

\paragraph{First exposure: compression.} Choose
\begin{equation*}
t=\left\lceil\log_4\frac{16(C_{\mathrm{tr}}d)^d}{3\varepsilon}\right\rceil,
\end{equation*}
and let $S\sim\mu$ and $V\sim\Bin(U,16/k)$ be independent. By the choice of $t$, Lemma~\ref{lem:fragment} gives
\begin{equation*}
\E_V\Pr_S[|M(S,V)|\ge t]\le\frac43(C_{\mathrm{tr}}d)^d4^{-t}\le\frac{\varepsilon}{4}.
\end{equation*}
By Markov's inequality, $\Pr_S[|M(S,V)|<t]\ge1/2$ with probability at least $1-\varepsilon/2$ over $V$. Call such a $V$ \emph{good}.

Fix a good $V$. Condition the distribution of $S$ on the event $|M(S,V)|<t$ and consider the induced probability distribution of $M(S,V)$. Its support consists of sets of size below $t$, and hence is $t$-bounded. For every nonempty $T\subseteq U$,
\begin{equation*}
\Pr_S[T\subseteq M(S,V)\mid |M(S,V)|<t]
\le2\Pr_S[T\subseteq S]
\le2k^{-|T|}
\le(k/2)^{-|T|}.
\end{equation*}
Here we used $M(S,V)\subseteq S$ and the fact that the conditioning event has probability at least $1/2$.

\paragraph{Second exposure: completion.}
Set $p_1=\min\{1,(2c/k)\log(2t/\varepsilon)\}$, where $c$ is the constant in Lemma~\ref{lem:bell}, and draw $W\sim\Bin(U,p_1)$ independently of $V$. For each good $V$, the induced fragment family is $t$-bounded and supports a $(k/2)$-spread probability measure, and hence is not $(2/k)$-small. Lemma~\ref{lem:bell}, applied with $a=2/k$ and $\delta=\varepsilon/2$, therefore shows that $W$ contains one of these fragments with probability at least $1-\varepsilon/2$. If $W$ contains a fragment $A$ in the support of this measure, choose $S$ with $M(S,V)=A$. Then $V\cup W$ contains the member of $\F$ selected in the definition of $M(S,V)$.

The probability that $V$ is not good is at most $\varepsilon/2$, and, conditional on any good $V$, the second exposure fails with probability at most $\varepsilon/2$. Hence $V\cup W$ contains a member of $\F$ with probability at least $1-\varepsilon$.

Since $V$ and $W$ are independent product measures,
\begin{equation*}
V\cup W\sim\Bin\left(U,1-\left(1-\frac{16}{k}\right)(1-p_1)\right),
\end{equation*}
whose density is at most $16/k+p_1$. Since $t=O(d\log(d+1)+\log(1/\varepsilon))$, we have $\log(2t/\varepsilon)=O(\log((d+1)/\varepsilon))$. Hence, for an absolute constant $C_1$,
\begin{equation*}
\frac{16}{k}+p_1\le\frac{C_1}{k}\log\frac{d+1}{\varepsilon}.
\end{equation*}
Taking the constant $C$ in the theorem at least $C_1$, the density of $V\cup W$ is at most $p$. Since the event of containing a member of $\F$ is increasing, the same conclusion holds for $\Bin(U,p)$.
\end{proof}

\section{The expectation-threshold theorem}\label{sec:thresholds}

We now derive Theorem~\ref{thm:main} from the spread theorem. For \(T\subseteq U\), write \(\F_T=\{S\setminus T:S\in\F,\ T\subseteq S\}\) for the \emph{link} of \(\F\) at \(T\), viewed as a family on \(U\setminus T\). For \(a>0\) and a family \(\mathcal A\subseteq2^U\), we call \(\sum_{A\in\mathcal A}a^{|A|}\) the \emph{\(a\)-cost} of \(\mathcal A\). Thus, for \(0<a\le1\), a family is \(a\)-small exactly when it has a cover of \(a\)-cost at most \(1/2\).

The proof starts by setting \(a\) to be a suitable constant multiple of \(q(\F)\), so that \(\F\) has a quantitative obstruction to \(a\)-smallness. We then use two structural ingredients. The first extracts local spread structure from a maximum packing: either the packing already gives a spread probability measure on \(\F\), or its saturated constraints form a cover of \(\F\) and the link at every saturated set supports a spread probability measure. The second uses the independent-sample estimate from Section~\ref{sec:weighted-traces} to bound the number of saturated sets of each size. This implies that the large saturated sets have negligible \(a\)-cost, allowing us to discard them while retaining a bounded-size family that is still not \(a\)-small. Finally, one random exposure using Lemma~\ref{lem:bell} locates one of these bounded-size saturated sets, and a second exposure uses the spread measure on its link to extend it to a member of \(\F\).

The packing formulation below is closely related to the linear-programming connection between fractional covers and spread measures; see, for example,~\cite[Proposition~1.5]{FKNP}. A related use of tight constraints arising from a maximal solution appears in the container argument of Campos and Samotij~\cite[Proposition~1.1(ii)]{CamposSamotij}. For $k\ge1$, consider the linear program
\begin{equation}\label{eq:packing-lp}
\begin{aligned}
\text{maximize}\quad
& \sum_{S\in\F} z_S \\
\text{subject to}\quad
& \sum_{S\in\F,\,T\subseteq S} z_S\le k^{-|T|}
\qquad (T\subseteq U), \\
& z_S\ge0
\qquad (S\in\F).
\end{aligned}
\end{equation}
The feasible region is nonempty and compact, so an optimal solution exists. Let $(z_S)_{S\in\F}$ be an optimal solution, and write $\alpha=\sum_{S\in\F}z_S$ for its total weight. The constraint corresponding to $T=\varnothing$ gives $\alpha\le1$. For $T\subseteq U$, write $z(T)=\sum_{S\in\F,\,T\subseteq S}z_S$. We call $T$ \emph{saturated} if $z(T)=k^{-|T|}$, and write $\mathcal T=\{T\subseteq U:z(T)=k^{-|T|}\}$ for the family of saturated sets.

The next lemma shows that either the maximum packing already gives a global spread measure, or the saturated sets cover $\F$ and each of their links supports a spread measure. 

\begin{lemma}\label{lem:saturated-links}
With the notation above, if $\alpha=1$, the weights $z_S$ define a $k$-spread probability measure supported on $\F$. If $\alpha<1$, then $\mathcal T$ is a cover of $\F$. Moreover, for every $T\in\mathcal T$, if $S\in\F$ containing $T$ is chosen with probability $z_S/z(T)$, then the resulting distribution of $S\setminus T$ is a $k$-spread probability measure supported on $\F_T$.
\end{lemma}

\begin{proof}
If $\alpha=1$, the weights $z_S$ form a probability measure, and the constraints in \eqref{eq:packing-lp} say exactly that it is $k$-spread.

Suppose now that $\alpha<1$. Then $\varnothing$ is not saturated, so $\varnothing\notin\mathcal T$. We first show that $\mathcal T$ covers $\F$. Indeed, if some \(S_0\in\F\) contained no saturated set, every constraint indexed by a subset of \(S_0\) would have positive slack. Since there are finitely many such constraints, \(z_{S_0}\) could be increased slightly while preserving all the constraints, contradicting the optimality of $(z_S)_{S\in\F}$. 

Fix $T\in\mathcal T$. Since $z(T)=k^{-|T|}>0$, the indicated distribution is well defined. For every nonempty $A\subseteq U\setminus T$,
\begin{equation*}
\Pr[A\subseteq S\setminus T]
=\frac{z(T\cup A)}{z(T)}
\le\frac{k^{-|T|-|A|}}{k^{-|T|}}
=k^{-|A|}.
\end{equation*}
Thus the induced distribution on $\F_T$ is $k$-spread.
\end{proof}

The next lemma is the point at which bounded VC dimension enters the packing argument. It uses the independent-sample estimate of Corollary~\ref{cor:independent-intersection} to bound the number of saturated sets of each size.

\begin{lemma}\label{lem:saturated-truncation}
With the notation above, suppose that $k\ge8$ and $\VC(\F)\le d$, where $d\ge1$ is an integer, and set $B=(C_{\mathrm{tr}}(d+1))^{d+1}$. Then, for every integer $m\ge1$,
\begin{equation}\label{eq:saturated-count}
|\{T\in\mathcal T:|T|=m\}|
\le B(4k)^m.
\end{equation}
\end{lemma}

\begin{proof}
Add the missing mass $1-\alpha$ at $\varnothing$. This gives a $k$-spread probability measure whose support is contained in $\F\cup\{\varnothing\}$, since adding mass at $\varnothing$ does not change the containment probabilities of nonempty sets. Its support has VC dimension at most $d+1$, since adjoining $\varnothing$ can increase VC dimension by at most one.

Let $S$ be sampled from this probability measure. For every nonempty $T\subseteq U$, we have $\Pr[T\subseteq S]=z(T)$. Therefore Corollary~\ref{cor:independent-intersection} gives, for every $m\ge1$,
\begin{equation*}
\sum_{T\subseteq U,\,|T|=m}z(T)^2
\le B\left(\frac4k\right)^m.
\end{equation*}
Since every saturated $T$ of size $m$ contributes $k^{-2m}$ to the sum on the left, \eqref{eq:saturated-count} follows.
\end{proof}

The two lemmas above capture the structural information supplied by the maximum packing. We now combine them with Lemma~\ref{lem:bell} and Theorem~\ref{thm:spread} to prove Theorem~\ref{thm:main}.

\begin{proof}[Proof of Theorem~\ref{thm:main}]
If $p=1$, the assertion is immediate. Hence assume $p<1$. Put $a=2q(\F)$ and $k=(32a)^{-1}$, and let $C_0$ be a constant for Theorem~\ref{thm:spread}. Since $p<1$, we have
$Cq(\F)\log((d+1)/\varepsilon)<1$. As $\log((d+1)/\varepsilon)\ge\log4$, taking the absolute constant $C$ in the statement of the theorem sufficiently large ensures that $a\le1/256$ and hence $k\ge8$.

\paragraph{A quantitative obstruction to $a$-smallness.}
We claim that every cover of \(\F\) has \(a\)-cost at least one. This is immediate for a cover containing \(\varnothing\). By the definition of \(q(\F)\) and continuity, every other cover has \(q(\F)\)-cost at least \(1/2\). Since every member \(T\) of such a cover is nonempty, \(a^{|T|}=(2q(\F))^{|T|}\ge 2q(\F)^{|T|}\). Thus its \(a\)-cost is at least twice its \(q(\F)\)-cost, and hence at least one. This proves the claim. In particular, \(\F\) is not \(a\)-small. The stronger lower bound of one gives enough room to discard a collection of sets of small total \(a\)-cost while retaining a family that is not \(a\)-small. 

\paragraph{The maximum packing.}
Apply the linear program \eqref{eq:packing-lp} with this \(k\), and let \((z_S)_{S\in\F}\), \(\alpha\), and \(\mathcal T\) be as defined above. If $\alpha=1$, Lemma~\ref{lem:saturated-links} gives a $k$-spread probability measure supported on a subfamily of $\F$, and hence with VC dimension at most $d$. Theorem~\ref{thm:spread} then gives the desired conclusion for a binomial random set of density at most $(C_0/k)\log((d+1)/\varepsilon)=32C_0a\log((d+1)/\varepsilon)$. Thus, taking the constant $C$ in the theorem at least $64C_0$ proves the theorem when $\alpha=1$.

We may therefore assume $\alpha<1$. By Lemma~\ref{lem:saturated-links}, $\mathcal T$ is a cover of $\F$, and every $T\in\mathcal T$ is nonempty. Moreover, $\F_T$ supports a $k$-spread probability measure and satisfies $\VC(\F_T)\le d$, since every set shattered by $\F_T$ is also shattered by $\F$.

\paragraph{Truncating the saturated cover.}
Let $B=(C_{\mathrm{tr}}(d+1))^{d+1}$ and $t=\lceil\log_8(8B)\rceil$. Since $4ak=1/8$, Lemma~\ref{lem:saturated-truncation} gives
\begin{equation*}
\sum_{T\in\mathcal T,\,|T|>t}a^{|T|}
\le \sum_{m>t} B(4k)^m a^m
=\sum_{m>t}B8^{-m}
\le\frac{B8^{-t}}7
\le\frac1{56}.
\end{equation*}
In particular, $t=O(d\log(d+1))$.

Let $\G=\{T\in\mathcal T:|T|\le t\}$. Although $\G$ need not itself cover $\F$, every cover of $\G$ has $a$-cost greater than $1/2$. Otherwise, such a cover together with the discarded saturated sets would cover $\mathcal T$, and hence $\F$, with total $a$-cost at most $1/2+1/56<1$, contradicting the preceding lower bound. Hence \(\G\) is a nonempty \(t\)-bounded family that is not \(a\)-small. 

\paragraph{Two random exposures.}
Set $p_0=\min\{1,ca\log(2t/\varepsilon)\}$, where $c$ is the constant in Lemma~\ref{lem:bell}, and let $V\sim\Bin(U,p_0)$. Since $\G$ is $t$-bounded and not $a$-small, Lemma~\ref{lem:bell} shows that, with probability at least $1-\varepsilon/2$, the set $V$ contains some $T\in\G$. On this event, choose one such $T$ according to a fixed order on $\G$.

Independently sample $W\sim\Bin(U,p_1)$, where $p_1=\min\{1,(C_0/k)\log(2(d+1)/\varepsilon)\}$. Conditional on $V$ and the selected $T$, the set $W\setminus T$ is a binomial random subset of $U\setminus T$ of density $p_1$. Since $\F_T$ supports a $k$-spread probability measure and has VC dimension at most $d$, Theorem~\ref{thm:spread} implies that, except with probability at most $\varepsilon/2$, the set $W\setminus T$ contains a member of $\F_T$. Thus $W\setminus T$ contains $S\setminus T$ for some $S\in\F$ with $T\subseteq S$. Since $T\subseteq V$, we have $S\subseteq V\cup W$.

The two failure probabilities sum to at most $\varepsilon$. Finally, since $t=O(d\log(d+1))$ and $k^{-1}=32a$, there is an absolute constant $C_1$ such that
\begin{equation*}
p_0+p_1\le C_1a\log\frac{d+1}{\varepsilon}
=2C_1q(\F)\log\frac{d+1}{\varepsilon}.
\end{equation*}
The union $V\cup W$ has distribution $\Bin(U,1-(1-p_0)(1-p_1))$, whose density is at most $p_0+p_1$. Taking $C\ge2C_1$, together with the preceding constant requirements, this density is at most $p$. Since the event of containing a member of $\F$ is increasing, a random set from $\Bin(U,p)$ contains a member of $\F$ with probability at least $1-\varepsilon$. This completes the proof.
\end{proof}

\paragraph{The transversal example.}
For sharpness, we use a standard transversal construction from the literature; see, for example, \cite{ALWZ,BCW}. Partition $U$ into pairwise disjoint blocks $U_1,\ldots,U_d$, each of size $L\ge2$, and let
\begin{equation*}
\F=\bigl\{\{x_1,\ldots,x_d\}:x_i\in U_i\text{ for every }i=1,\ldots,d\bigr\}.
\end{equation*}
Thus $\F$ is the family of all transversals of the partition. The uniform probability measure on $\F$ is $L$-spread: every partial transversal $T$ is contained in an $L^{-|T|}$ fraction of the members of $\F$, while a set containing two points from the same block is contained in none. Moreover, $\VC(\F)=d$, since a set consisting of one point from each block is shattered, whereas every set of size $d+1$ contains two points from the same block and hence cannot be shattered.

A random set contains a member of $\F$ exactly when it meets every block, so for $W\sim\Bin(U,p)$ we have
\begin{equation}\label{eq:sharpness}
\Pr_W[\exists S\in\F:S\subseteq W]
=\bigl(1-(1-p)^L\bigr)^d.
\end{equation}
Thus, for success probability $1-\varepsilon$, the least possible $p$ is $1-\bigl(1-(1-\varepsilon)^{1/d}\bigr)^{1/L}$.
Since $1-(1-\varepsilon)^{1/d}=\Theta(\varepsilon/d)$ for $0<\varepsilon\le1/2$, and $1-e^{-x}=\Theta(x)$ for $0\le x\le1$, the least such $p$ is $\Theta(L^{-1}\log((d+1)/\varepsilon))$ whenever $L^{-1}\log((d+1)/\varepsilon)\le1$. Taking $k=L$, this shows that the dependence on $k$, $d$ and $\varepsilon$ in Theorem~\ref{thm:spread} is optimal up to absolute constants.

We next compute \(q(\F)\) and \(q_f(\F)\), showing that \(q(\F)=q_f(\F)=2^{-1/d}/L\). Using all $L^d$ members of $\F$ as an integral cover gives $\sum_{S\in\F}p^{|S|}=(Lp)^d$, and hence $q(\F)\ge2^{-1/d}/L$. For the matching upper bound on \(q_f(\F)\), let $\lambda:2^U\to[0,\infty)$ be an arbitrary fractional cover of $\F$, so that $\sum_{T\subseteq S}\lambda(T)\ge1$ for every $S\in\F$. We may assume that $\lambda(T)=0$ whenever $T$ is contained in no member of $\F$, since such weights do not contribute to any covering inequality. Thus every $T$ with $\lambda(T)>0$ is a partial transversal, and in particular $|T|\le d$.

Let $S$ be a uniformly random member of $\F$. For every $T$ with $\lambda(T)>0$, we have $\Pr_S[T\subseteq S]=L^{-|T|}$. Averaging the fractional-cover inequalities over $S$ therefore gives
\begin{equation*}
1\le \E_S\sum_{T\subseteq S}\lambda(T)
=\sum_{T\subseteq U}\lambda(T)\Pr_S[T\subseteq S]
=\sum_{T\subseteq U}\lambda(T)L^{-|T|}.
\end{equation*}
Since $0\le |T|\le d$ whenever $\lambda(T)>0$, we have $(Lp)^{|T|}\ge\min\{1,(Lp)^d\}$. Consequently,
\begin{equation*}
\sum_{T\subseteq U}\lambda(T)p^{|T|}
=\sum_{T\subseteq U}\lambda(T)L^{-|T|}(Lp)^{|T|}
\ge\min\{1,(Lp)^d\}\sum_{T\subseteq U}\lambda(T)L^{-|T|}
\ge\min\{1,(Lp)^d\}.
\end{equation*}
If $p>2^{-1/d}/L$, then the right-hand side is greater than $1/2$, so $q_f(\F)\le2^{-1/d}/L$. Together with $q(\F)\ge2^{-1/d}/L$ and $q(\F)\le q_f(\F)$, this yields $q(\F)=q_f(\F)=2^{-1/d}/L$. Together with the preceding estimate for the least density achieving success probability $1-\varepsilon$, this also shows that the high-probability form of Theorem~\ref{thm:main} is optimal up to absolute constants whenever $L^{-1}\log((d+1)/\varepsilon)\le1$.

Finally, take $L=d$ and let $d\to\infty$. Taking $\varepsilon=1/2$ in \eqref{eq:sharpness} gives $p_c(\F)=1-(1-2^{-1/d})^{1/d}$. Since $1-2^{-1/d}=(1+o(1))\log 2/d$, we have $\log(1/(1-2^{-1/d}))=\log d+O(1)$, and hence $p_c(\F)=(1+o(1))\log d/d$. On the other hand, $q(\F)=q_f(\F)=2^{-1/d}/d=(1+o(1))/d$. Therefore $p_c(\F)/q(\F)=(1+o(1))\log d$. This shows that the logarithmic dependence on $d$ in Theorem~\ref{thm:main} is necessary up to absolute constants. On the other hand, this example does not establish the necessity of the $\log(d+1)$ factor in Corollary~\ref{cor:integral-fractional}, since here $q(\F)=q_f(\F)$. Indeed, Talagrand's integral--fractional conjecture~\cite{Talagrand} predicts that this dependence can be removed altogether.

\section{Sunflowers}\label{sec:applications}

We call a nonempty family $\F\subseteq2^U$ \emph{$k$-spread} if the uniform probability measure on $\F$ is $k$-spread. To prove Corollary~\ref{cor:robust}, we use the maximal-core reduction from sunflower arguments; see, for example,~\cite{ALWZ}. It converts the size assumption into spreadness of a residual family, and Theorem~\ref{thm:spread} then gives the robust containment property.

\begin{proof}[Proof of Corollary~\ref{cor:robust}]
Let $C_{\mathrm{sp}}\ge1$ be a constant for which Theorem~\ref{thm:spread} holds. Set $k=C_{\mathrm{sp}}p^{-1}\log((d+1)/\varepsilon)$. We prove the assertion with $C=C_{\mathrm{sp}}$. Choose an inclusion-maximal set $T\subseteq U$ that is contained in at least $|\F|k^{-|T|}$ members of $\F$, and put
\begin{equation*}
\G=\{F\setminus T:F\in\F,\ T\subseteq F\}.
\end{equation*}
Such a set $T$ exists because $\varnothing$ is contained in all $|\F|$ members of $\F$. By the assumed bound, $|\F|>k^n$, and hence $|\F|k^{-n}>1$. Therefore $|T|<n$, since an $n$-set is contained in at most one member of the $n$-uniform family $\F$. Thus every member of $\G$ is nonempty.

For every nonempty $A\subseteq U\setminus T$, the maximality of $T$ gives
\begin{equation*}
|\{G\in\G:A\subseteq G\}|<|\F|k^{-|T|-|A|}\le|\G|k^{-|A|}.
\end{equation*}
Thus $\G$ is $k$-spread. Moreover, every set shattered by $\G$ is also shattered by $\F$, so $\VC(\G)\le d$.

Since $k>1$, no point belongs to every member of $\G$, and hence $\bigcap_{G\in\G}G=\varnothing$. As every member of $\G$ is nonempty, $\G$ has at least two members. By the definition of $k$, we have $p=(C_{\mathrm{sp}}/k)\log((d+1)/\varepsilon)$. Therefore Theorem~\ref{thm:spread}, applied on the ground set $U\setminus T$, shows that a random set $W\sim\Bin(U\setminus T,p)$ contains a member of $\G$ with probability at least $1-\varepsilon$. It follows that $\mathcal R=\{F\in\F:T\subseteq F\}$ has common intersection exactly $T$. Since every member of $\G$ is nonempty, we also have $T\notin\mathcal R$. Hence $\mathcal R$ is a $(p,\varepsilon)$-robust sunflower with core $T$.
\end{proof}

A random-coloring argument converts a robust sunflower into an ordinary sunflower. 

\begin{corollary}\label{cor:sunflower}
There is an absolute constant $C>0$ such that, for integers $d,n\ge1$ and $r\ge2$, every $n$-uniform family $\F$ with $\VC(\F)\le d$ and $|\F|>(Cr\log(d+1))^n$ contains an $r$-sunflower.
\end{corollary}

\begin{proof}
Apply Corollary~\ref{cor:robust} with $p=1/(2r)$ and $\varepsilon=1/2$, enlarging the absolute constant $C$ if necessary, to obtain a $(1/(2r),1/2)$-robust sunflower $\mathcal R$ with core $T$. Color the points of $U\setminus T$ independently and uniformly with $2r$ colors. Each color class contains a petal with probability at least $1/2$, so the expected number of successful classes is at least $r$. Hence some coloring has at least $r$ successful classes, and choosing one petal from each of them gives $r$ pairwise disjoint petals. Adding $T$ yields an $r$-sunflower in $\F$.
\end{proof}

A \emph{matching} is a family of pairwise disjoint sets. Theorem~\ref{thm:spread} also immediately implies the corresponding matching statement by the same random-coloring argument: there is an absolute constant \(C>0\) such that, for integers \(d,r\ge1\), every nonempty \(k\)-spread family \(\F\subseteq2^U\setminus\{\varnothing\}\) with \(\VC(\F)\le d\) and \(k\ge Cr\log(d+1)\) contains \(r\) pairwise disjoint members.

\section*{Acknowledgements}

The author is supported by the National Natural Science Foundation
of China under Grant Nos.~12571352 and 12231014, and the Fundamental
Research Funds for the Central Universities.

During the preparation of this manuscript, the author used ChatGPT6 Pro and ChatGPT5.6 Sol as auxiliary tools for exploratory discussion, checking calculations, and improving the exposition. All mathematical statements, arguments, proofs, and references were independently verified by the author, who assumes full responsibility for the content of the manuscript.

\end{document}